\newcommand{\Hmm}[1]{\leavevmode{\marginpar{\tiny%
$\hbox to 0mm{\hspace*{-0.5mm}$\leftarrow$\hss}%
\vcenter{\vrule depth 0.1mm height 0.1mm width \the\marginparwidth}%
\hbox to
0mm{\hss$\rightarrow$\hspace*{-0.5mm}}$\\\relax\raggedright #1}}}
\documentclass[12pt]{article}
\usepackage[T1]{fontenc}
\usepackage[utf8]{inputenc}
\usepackage[english]{babel}
\usepackage{graphicx}
\usepackage[font=small]{quoting}
\usepackage{caption}
\usepackage{mathrsfs}
\usepackage[top=.8in,bottom=1in,left=1in,right=1in]{geometry}
\usepackage[english]{babel}
\usepackage{stmaryrd}
\usepackage{verbatim}
\usepackage{color}
\usepackage{picture}
\usepackage{amsmath,amsthm,amsfonts,amssymb}
\usepackage{comment}
\usepackage{mathtools}
\usepackage{pgf,tikz}
\usepackage{titling}

\usetikzlibrary{arrows}

\newtheorem{thm}{Theorem}[section]
\newtheorem{lem}[thm]{Lemma}

\theoremstyle{definition}

\theoremstyle{remark}

\DeclareMathOperator{\Qp}{\mathcal Q_{p}}

\begin{document}

\title{Eigenvalues of the  Finsler $p$-Laplacian \\on varying domains. }
\author{{\ } Giuseppina di Blasio\thanks{Universit\`{a} degli Studi della Campania \textquotedblleft Luigi
Vanvitelli\textquotedblright, Dipartimento di Matematica e Fisica, Viale Lincoln, 5 81100 Caserta, Italy. Email:
giuseppina.diblasio@unicampania.it} {, } Pier Domenico Lamberti \thanks{Universit\`{a} degli studi di Padova, Dipartimento di Matematica \textquotedblleft Tullio Levi-Civita \textquotedblright, Via Trieste, 63 35121 Padova, Italy. Email:
lamberti@math.unipd.it }}
\maketitle

\begin{abstract} We study the dependence of the first eigenvalue of the Finsler $p$-Laplacian and the corresponding eigenfunctions upon perturbation of the domain and we generalize a few results known for the standard $p$-Laplacian. In particular, we prove a Frech\'{e}t differentiability result for the eigenvalues, we compute the corresponding Hadamard formulas and we prove a continuity result for the eigenfunctions. Finally, we briefly discuss a well-known overdetermined problem and we show how to deduce
the Rellich-Pohozaev identity for  the Finsler $p$-Laplacian  from the Hadamard formula.
\end{abstract}

\vspace{11pt}

\noindent
{\bf Keywords:} Finsler, anisotropic $p$-Laplacian, stability of eigenvalues,
domain perturbation.

\vspace{6pt}
\noindent
{\bf 2000 Mathematics Subject Classification:} 35P15, 35J25, 47A75, 47B25.

\bigskip
\section{Introduction}

Let $F$ be a positive, one-homogeneous, convex function in $\mathbb{R}^{N}$ and $p\in ]1,+\infty [$. In this paper we deal with a class of operators of the form
\begin{equation}\label{Qp}
\Qp u:=\text{div}(F^{p-1}(Du)F_{\xi}(Du))   ,
\end{equation}
acting on real-valued functions $u$ defined on an open connected subset $\Omega$ of ${\mathbb{R}}^N$,  (by $F_{\xi} $ we denote the gradient of $F$). This class includes the standard $p$-Laplacian
$$ \sum_{i=1}^N \frac{\partial}{\partial x_i} \left( |Du|^{p-2}\frac{\partial u}{\partial x_i}  \right) $$  which is  obtained when $F$ is the Euclidean modulus, as well as the pseudo-$p$-Laplacian
$$ \sum_{i=1}^N \frac{\partial}{\partial x_i} \left( \left| \frac{\partial u}{\partial x_i}    \right|^{p-2}\frac{\partial u}{\partial x_i}  \right) $$
which is obtained when $F(\xi)=(\sum_{i}^N|\xi_i|^p)^{1/p}$ for all $\xi\in \mathbb{R}^N$.     The operator in \eqref{Qp}  is often called anisotropic
$p$-Laplacian or Finsler $p$-Laplacian and  has been studied in several papers, see e.g., \cite{aflt, bfk, BKJ, ciraolo, CS, DdBG, kn08, wxpac}.
 Clearly, operator \eqref{Qp} is of interest mainly for $N\geq 2$ since for $N=1$ one obtains only  the standard $p$-Laplacian. Thus, also for the purposes of the paper itself, we shall tacitly assume throughout the paper that $N\geq 2$, although, in principle, this is not really required.
Here we consider the eigenvalue problem associated with $\Qp$
\begin{equation}
\left\{
\begin{array}{ll}
-\Qp u =\lambda |u|^{p-2}u, & \text{in }\Omega, \\
u=0, & \text{on }\partial \Omega ,
\end{array}
\right.  \label{intro_eigpb}
\end{equation}
 where $\Omega $ is assumed to be of class $C^{1,\alpha }$, $\alpha \in (0,1]$. It is well known that
 the eigenvalue problem \eqref{intro_eigpb} admits a unique eigenvalue $\lambda _{p}(\Omega )$ associated with a positive eigenfunction, which is simple and isolated, see e.g., \cite{bfk,dgp2}.
%Moreover, $\lambda _{p}(\Omega )$ has the
%following well-known variational characterization:
%\begin{equation}
%\lambda _{p}(\Omega )=\min_{\varphi \in W_{0}^{1,p}(\Omega )\setminus
%\{0\}}\frac{\displaystyle\int_{\Omega }F^{p}(D \varphi )\ dx}{
%\displaystyle\int_{\Omega }|\varphi |^{p}\ dx}.  \label{rayleigh}
%\end{equation}

Our aim is to analyze the dependence of  $\lambda _{p}(\Omega )$ on $\Omega$. To do so, we adopt  the approach developed in  \cite{Lamb} for the standard $p$-Laplacian and we extend the corresponding results to the general case.  Namely, we consider a suitable class of domains which are represented as diffeomorphic images of $\Omega$  via a diffeomorphism $\phi:\Omega \to \phi (\Omega)$ of class $C^{1,\alpha}$, we study  the map which takes $\phi$ to the eigenvalue $\lambda_p(\phi(\Omega ))$,   we prove a Frech\'et differentiability result for dependence of  $\lambda_p(\phi(\Omega ))$ on $\phi$, and we compute the corresponding Hadamard  formula for  the Frech\'et differential, see Theorems~\ref{Frechet_diff} and \ref{Th_Hadamard}.  We also prove that the uniquely determined positive normalized eigenfunction $u_{p,\phi}$ associated with $\lambda_p(\phi (\Omega))$ depends with continuity on $\phi$ in the sense that its pull-back $u_{p,\phi}\circ \phi$, which is a function defined in the fixed domain $\Omega$, varies continuously in $C^1(\overline{\Omega})$ upon variation of $\phi$, see Theorem~\ref{Th_contuinity}.

 Finally, we briefly discuss how our results provide a  motivation for the study of the well-known overdetermined problem  \eqref{overdet} associated with the Faber-Krahn inequality for the Finsler $p$-Laplacian and we show how to deduce the Rellich-Pohozahev identity \eqref{rellichformula} from the Hadamard formula.

We note that domain perturbation problems have been extensively studied in the
Euclidean case for the Dirichlet Laplacian as well as for more general elliptic operators and there is a wide related literature, see e.g., \cite{arrieta, BL, buosopoly, Lamb, N, S} and the references therein. In particular, we refer to the  monograph  \cite{henry}  where the method of domain transplantation via diffeomorphisms is extensively applied and to \cite{buoso} for a collection of Hadamard-type  formulas.  For  the case of quasi-linear PDEs, much less is available in the literature and we refer to \cite{burlam, GS, Lamb} for the analysis of the standard $p$-Laplacian
and to \cite{BLhardy} for a  $p$-Laplacian problem involving  a Hardy potential.

\section{Preliminaries and notation}

Throughout the paper we will consider a convex even 1-homogeneous function
\begin{equation*}
\mathbb{R}^{N}\owns \xi\mapsto F(\xi)\in [0,+\infty[,
\end{equation*}
that is, a convex function such that
\begin{equation}  \label{eq:omo}
F(t\xi)=|t|F(\xi),
\end{equation}
for all $t\in \mathbb{R},\, \xi \in \mathbb{R}^{N}$ and we assume that
\begin{equation}  \label{eq:lin}
a|\xi| \le F(\xi),
\end{equation}
for all $ \xi \in \mathbb{R}^{N}$, for some constant $a>0$. The hypotheses on $F$ imply there exists $b\ge a$
such that
\begin{equation*}
F(\xi)\le b |\xi| ,
\end{equation*}
for all $\xi \in \mathbb{R}^{N}$. Moreover, throughout the paper we will assume that $p\in ]1,\infty[$ and that  $F\in C^{2}(\mathbb{R}^{N}\setminus \{0\})$, and
\begin{equation}  \label{strong}
[F^{p}]_{\xi\xi}(\xi)\text{ is positive definite in } \mathbb{R}^{N}\setminus\{0\}.
\end{equation}

The hypothesis \eqref{strong}  ensures that the operator $\Qp $ defined in \eqref{Qp} is elliptic, hence there exists a positive constant $\gamma $ such that
\begin{equation}
\sum_{i,j=1}^{n}\frac{\partial }{\partial \xi _{j}}({F^{p-1}(\eta )F}_{\xi
_{i}}\left( \eta \right) ){\xi _{i}\xi _{j}}\geq \gamma |\eta |^{p-2}|\xi
|^{2},  \label{ellip}
\end{equation}
for all $\eta \in \mathbb{R}^{n}\setminus \{0\}$ and for all $\xi \in
\mathbb{R}^{n}$. %\begin{rem}
%We stress that for $p\ge 2$ the condition
%\begin{equation*}
%\nabla^{2}_{\xi}[F^{2}](\xi)\text{ is positive definite in } \R^{N}\setminus\{0\},
%\end{equation*}
%implies \eqref{strong}.
%\end{rem}
We recall that the {\it polar} function $F^{o}\colon \mathbb{R}^{N}\rightarrow \lbrack 0,+\infty
\lbrack $ of $F$ is defined by
\begin{equation*}
F^{o}(v)=\sup_{\xi \in \mathbb{R}^N\setminus \{0\}}\frac{\langle \xi ,v\rangle }{F(\xi )},
\end{equation*}
for all $v\in \mathbb{R}^N$, where $\langle \xi ,v\rangle $ denotes the scalar product of $\xi$ and $v$.
It is easy to verify that also $F^{o}$ is a convex function which satisfies
properties \eqref{eq:omo} and \eqref{eq:lin}. Furthermore,
\begin{equation*}
F(v)=\sup_{\xi \in \mathbb{R}^N\setminus \{0\}  }\frac{\langle \xi ,v\rangle }{F^{o}(\xi )},
\end{equation*}
for all $v\in \mathbb{R}^N$.
From the above property it holds that
\begin{equation*}
|\langle \xi ,\eta \rangle |\leq F(\xi )F^{o}(\eta ),\qquad \forall \xi
,\eta \in \mathbb{R}^{N}.  \label{imp}
\end{equation*}
The set
\begin{equation*}
\mathcal{W}=\{\xi \in \mathbb{R}^{N}\colon F^{o}(\xi )<1\}
\end{equation*}
is called {\it Wulff shape} centered at the origin. We put $\kappa _{N}=|\mathcal{W}|$, where $|\mathcal{W}|$
denotes the Lebesgue measure of $\mathcal{W}$. More generally, we denote with $\mathcal{W}_{r}(x_{0})$ the
set $x_0+r\mathcal{W}$, that is the Wulff shape centered at $x_{0}$ with
measure $\kappa _{N}r^{N}$, and $\mathcal{W}_{r}(0)=\mathcal{W}_{r}$.

%We observe that $F$ is the support function  of $\overline{\mathcal W}$.

We also recall  the following properties of $F$ and $F^o$ (see e.g. \cite{BP}):
\begin{gather*}  \label{terza}
\langle F_\xi(\xi) , \xi \rangle= F(\xi), \quad \langle F_\xi^{o} (\xi), \xi
\rangle = F^{o}(\xi),\qquad \forall \xi \in \mathbb{R}^N\setminus \{0\} \\
F( F_\xi^o(\xi))=F^o( F_\xi(\xi))=1,\quad \forall \xi \in \mathbb{R}
^N\setminus \{0\}, \\
F^o(\xi) F_\xi( F_\xi^o(\xi) ) = F(\xi) F_\xi^o\left( F_\xi(\xi) \right) =
\xi\qquad \forall \xi \in \mathbb{R}^N\setminus \{0\}.
\end{gather*}

%\subsection{The first Dirichlet eigenvalue for $\mathcal{Q}_p$}

%Let $\alpha \in (0,1]$ and $\Omega $ be a bounded connected open subset of $\mathbb{R}^{N}$, $N\geq 2$, of class $C^{1,\alpha }$,  and  $p\in ]1,\infty [ $.

Let  $\Omega $ be a bounded open subset of $\mathbb{R}^{N}$   and  $p\in ]1,\infty [ $.
We recall that  the eigenvalue problem  \eqref{intro_eigpb} has to be considered  in the weak form, which means that the unknown $u$ belongs to the standard Sobolev space $W^{1,p}_0(\Omega )$ and satisfies the equation
\begin{equation}  \label{eigpb}
\int_{ \Omega  }F^{p-1}\left( Du\right) F_{\xi }\left(
Du\right) \cdot D\varphi ^{T}dy=\lambda \int_{ \Omega \,
}|u|^{p-2}u\varphi dy\, ,
\end{equation}
for all $\varphi \in W^{1,p}_0(\Omega )$.
We observe that in this paper the elements of $\mathbb{R}^N$ are thought as row vectors and that  the inverse and transpose
of a matrix $M$ are denoted by  $M^{-1}$ and $M^{T}$ respectively.

Note that if $u\in W^{1,p}_0(\Omega )$ is a non-trivial solution to the previous equation, then choosing $\varphi =u$ in \eqref{eigpb} and using  Euler's homogeneous function Theorem, yield the equality $\lambda =\int_{\Omega }F^{p}(D u )\ dx /  \int_{\Omega }| u |^{p}\ dx$ which suggests that the smallest eigenvalue, denote by
%\begin{equation} \left\{ \begin{array}{ll}-\Qp u=\lambda |u|^{p-2}u & \text{in }\Omega \\u=0 & \text{on }\partial \Omega .\end{array}
%\right.  \label{eigpb}
%\end{equation}
 $\lambda _{p}(\Omega )$, is positive and has the following
  variational characterisation
\begin{equation}
\lambda _{p}(\Omega )=\min_{u \in W_{0}^{1,p}(\Omega )\setminus
\{0\}}\frac{\displaystyle\int_{\Omega }F^{p}(D u )\ dx}{
\displaystyle\int_{\Omega }| u |^{p}\ dx}.  \label{rayleigh}
\end{equation}
In fact, the following known results hold, see e.g.,  \cite
{bfk,dgp2}.

\begin{thm}
If $p>1$  and  $\Omega$ is a bounded open set in $\mathbb{R}^{N}$  there
exists a function $u_{1}\in W^{1,p}_0(\Omega)$
which achieves the minimum in \eqref{rayleigh}, and satisfies the problem
\eqref{eigpb} with $\lambda=\lambda _{p}(\Omega )$. Moreover, if $\Omega$ is
connected, then $\lambda _{p}(\Omega )$ is simple, that is, the corresponding
eigenfunction is unique up to a multiplicative constant, and the first
eigenfunction has constant sign in $\Omega$. Finally, if in addition $\Omega$ is of class $C^{1,\alpha}$ with $\alpha \in ]0,1]$ then
$u_1$ is of class $C^{1,\beta}(\bar \Omega )$ for some $\beta \in ]0,\alpha ]$.
\end{thm}

We recall  the scaling and monotonicity properties of $
\lambda _{p}(\Omega )$ in the following lemma the proof of which can be carried out as in the case of the classical $p$-Laplacian.

\begin{lem}\label{monoto}
Let $\Omega$ be a bounded open set in $\mathbb{R}^{N}$.  Then the
following properties hold.

\begin{enumerate}
\item For $t>0$ it holds $\lambda_{p}
(t\Omega)=t^{-p}\lambda_{p}(\Omega) $.

\item If  $\tilde \Omega $ is an open set contained in $\Omega$  then $\lambda_{p}
(\tilde \Omega)\ge \lambda_{p} (\Omega)$.

\item If $1<p<s<+\infty $ then $p[\lambda _{p}(\Omega
)]^{1/p}\leq s[\lambda _{s}(\Omega )]^{1/s}$.
\end{enumerate}
\end{lem}

  \section{Domain perturbation and stability of eigenfunctions}
\bigskip  In order to
study the dependence of $\lambda _{p}(\Omega )$ with respect to the variation of the domain $
\Omega$, we consider problem (\ref{eigpb}) in a class of domains which
are parameterized by means of diffeomorphisms defined on a  domain $\Omega $.
 More precisely, we fix a bounded  domain (i.e. a bounded connected open set) $\Omega$ of class $C^{1,\alpha}$ with $\alpha \in (0,1]$  and we set
\begin{equation*}
\Phi _{1,\alpha }\left( \Omega \right) \equiv \left\{ \phi \in C^{1,\alpha }
(\overset{\_}{\Omega },\mathbb{R}^{N}):\phi \text{ is injective and det}D\phi
\left( x\right) \neq 0\text{, for all }x\in \overset{\_}{\Omega }\right\} .
\end{equation*}

\bigskip The space  $C^{1,\alpha }
(\overset{\_}{\Omega },\mathbb{R}^{N})$ is endowed with its standard norm defined by
$$
\|  \Phi \|_{  C^{1,\alpha }
(\overset{\_}{\Omega },\mathbb{R}^{N})}=    \max\{\| \phi \|_{\infty }, \| \nabla \phi \|_{\infty} , | \nabla\phi   |_{\alpha}\} \, ,$$
where $|\cdot |_{\alpha}$ denotes the usual H\"older semi-norm.

The set $\Phi _{1,\alpha }\left( \Omega \right) $ is an open subset
of $C^{1,\alpha }(\overset{\_}{\Omega },\mathbb{R}^{N})$;
moreover, also $\phi \left( \Omega \right) $ is a bounded connected
open subset of $\mathbb{R}^{N}$ of class $C^{1,\alpha }$ and $
\partial \phi \left( \Omega \right) =\phi \left( \partial \Omega \right) $
for all $\phi \in \Phi _{1,\alpha }\left( \Omega \right)$, see \cite{LL} and  \cite{Lanza} for details.

For $\phi \in \Phi _{1,\alpha }\left( \Omega \right)$,  we consider  the smallest
eigenvalue  $\lambda_p(\phi (\Omega))$    of problem   \eqref{eigpb} in the domain $\phi (\Omega)$ and we set
$$\lambda _{p,\phi}=\lambda_p(\phi (\Omega)).$$

\begin{comment}
\begin{equation}
\left\{
\begin{array}{ll}
-\Qp u=\lambda |u|^{p-2}u & \text{in }\phi \left( \Omega \right) \\
u=0 & \text{on }\partial \phi \left( \Omega \right) ,
\end{array}
\right.  \label{eigpb_Phi}
\end{equation}
 denoted by $\lambda _{p,\phi}$. We say that $u\in
W_{0}^{1,p}\left( \phi \left( \Omega \right) \right) $ satisfies the
eigenvalue problem (\ref{eigpb_Phi}) if

\begin{equation*}\label{egvalue}
\int_{\phi \left( \Omega \right) }F^{p-1}\left( Du\right) F_{\xi }\left(
Du\right) \cdot D\varphi ^{T}dy=\lambda \int_{\phi \left( \Omega \right)
}|u|^{p-2}u\varphi dy
\end{equation*}

for all $\varphi \in W_{0}^{1,p}\left( \phi \left( \Omega \right) \right)$. We also say that $u$ is an eigenfunction of $\lambda$.
\end{comment}
We observe that
since the open set $\phi \left( \Omega \right) $ has a regular boundary, the corresponding
eigenfunctions attain the boundary value zero in the classical sense. It
follows that there exists a uniquely determined eigenfunction of $\lambda
_{p,\phi}$, denoted by $\ u_{p,\phi }$, such
that
\begin{equation}\label{eigen_cond}
\int_{\phi \left( \Omega \right) }|u_{p,\phi }\ |^{p}dy=1 \quad \text{and} \quad u_{p,\phi }\ (y)>0\quad \forall y\in \phi
\left( \Omega \right) ,
\end{equation}
and such that $u_{p,\phi }\in C^{1,\beta }(\overline{\phi (\Omega)})$ for some $\beta \in ]0,\alpha ]$.     In particular we
have that
\begin{equation}\label{eig}
\lambda _{p,\phi}=\min_{\varphi \in C^1_{0}(\overline{\phi (\Omega)} )\setminus
\{0\}}\frac{\displaystyle\int_{\phi(\Omega) }F^{p}(D\varphi )\ dy}{
\displaystyle\int_{\phi(\Omega) }|\varphi |^{p}\ dy}=\int_{\phi(\Omega) }F^{p}\left(Du_{p,\phi }\ \right) dy,
\end{equation}
where  $C^1_{0}(\overline{\phi (\Omega)} )$ denotes the space of functions in $C^1(\overline{\phi (\Omega)} )$ which vanish at the boundary of $\Omega$.

We plan to study the differentiability of the map which takes $\phi \in \Phi
_{1,\alpha }\left( \Omega \right) $ to $\lambda _{p,\phi}\in \mathbb{R}$. To do so, we first analyze the dependence of
the eigenfunction corresponding to $\lambda _{p,\phi}$, with respect to $\phi$. We stress the fact that the domain of $u_{p,\phi }
$ changes when $\phi $ is perturbed, hence we will not consider $u_{p,\phi }$
itself but its pull-back
\[
v_{p,\phi }\equiv u_{p,\phi }\circ \phi ,
\]
which is defined in the fixed domain $\Omega$.

\bigskip
Using a change of variable we rewrite problem (\ref{eigpb}) on $\phi \left( \Omega \right) $ into an eigenvalue problem defined in $\Omega $
which is solved by $v_{p,\phi }$. More precisely, for $\phi
\in \Phi _{1,\alpha }\left( \Omega \right) $ and $\lambda \in \left(
0,+\infty \right)$, it is easy to see that  a function $u\in W_{0}^{1,p}\left( \phi \left( \Omega
\right) \right) $ satisfies the eigenvalue problem (\ref{eigpb}) on $\phi (\Omega)$ if and only
if the function $v:= u\circ \phi $ belongs to $W_{0}^{1,p}\left( \Omega
\right) $ and satisfies the equation
\begin{eqnarray}\lefteqn{
\int_{\Omega }F^{p-1}\left( Dv\cdot (D\phi )^{-1}\right) F_{\xi }\left(
Dv\cdot (D\phi )^{-1}\right) \cdot \left[ (D\phi )^{-1}\right] ^{T}\cdot
D\varphi ^{T}|\det D\phi |dx}\nonumber \\
& & \qquad \qquad \qquad \qquad \qquad \qquad \qquad \qquad \qquad \qquad=\lambda \int_{\Omega }|v|^{p-2}v\varphi |\det
D\phi |dx,  \label{Weak_sol}
\end{eqnarray}
for all \bigskip $\varphi \in W_{0}^{1,p}\left( \Omega \right)$. Moreover,
\begin{equation*}
\lambda _{p,\phi}=\int_{\Omega }F^{p}\left(
Dv_{p,\phi }\cdot (D\phi )^{-1}\ \right) |\det D\phi |dx .
\end{equation*}
In what follows we denote the inverse
matrix of $D\phi$ with $J,$ that is $J=(D\phi )^{-1}.$

The following result on the continuous  dependence of the function
$v_{p,\phi }$ on $\phi \in \Phi _{1,\alpha }\left( \Omega \right) $ holds.

\begin{thm}
\label{Th_contuinity}Let $p>1$ and $\Omega $ be a bounded domain
in  $\mathbb{R}^{N}$ of class $C^{1,\alpha }$ with $\alpha \in
\left( 0,1\right]$. The function which maps $\phi \in \Phi _{1,\alpha
}\left( \Omega \right) $ to $v_{p,\phi }\in
C^{1}(\overline{\Omega })$ is continuous.
\end{thm}

\begin{proof}

Let $\tilde{\phi }\in\Phi _{1,\alpha }\left( \Omega \right) $ be fixed, and let us consider a sequence $\phi _{n}$, $n\in \mathbb{N}$,  which converges to $\tilde{\phi}$ in $\Phi _{1,\alpha }\left( \Omega \right)$. We want to prove
that there exists a subsequence $\phi _{n_{m}}$, $m\in {\mathbb{N}}$,  of $\phi _{n}$ such that $
v_{p,\phi _{n_{m}}}$ converges to $v_{p,\tilde{\phi }}$ in $C^{1}
(\overline{\Omega }).$ First of all we observe that $u_{p,\phi_n }$ is bounded.
Indeed, see \cite{dgp2}, there exists a constant $C=C\left(
N,p,F\right) $ such that
\begin{equation}
\left\Vert u_{p,\phi _{n}}\right\Vert _{\infty }\leq C\lambda _{p,\phi
_{n}}\left\Vert u_{p,\phi _{n}}\right\Vert _{1}.  \label{L^infty}
\end{equation}

Let $W$ be bounded open set (for example, a ball or a Wulf shape) with closure striclty contained in $\tilde\phi (\Omega)$. As in \cite[Prop.~2.1]{Lamb}, by the strong convergence of $\phi_n$ to $\tilde \phi$ one can see that $W$ is contained in $\phi_n(\Omega)$ for all $n$ sufficiently large.
Recalling that the eigenvalues are monotone with respect to domain inclusion, we get
\begin{equation}\label{eig_upb}
\lambda _{p,\phi_{n}}\leq \lambda _{p} (W),
\end{equation}
for all $n$ sufficiently large.
By (\ref{L^infty}), \eqref{eig_upb}, H\"{o}lder inequality and \eqref{eigen_cond} we get
\begin{equation*}
\left\Vert u_{p,\phi _{n}}\right\Vert _{\infty }\leq
C_1 ,
\end{equation*}
with $C_1>0$. In particular
\begin{equation}
\sup_{n\in \mathbb{N}}\left\Vert v_{p,\phi _{n}}\right\Vert _{\infty } \ne
\infty .  \label{L^infty_v}
\end{equation}

By rewriting \eqref{Weak_sol} in a simpler form,  we see that functions $v_{p,\phi _{n}}$ satisfy the equation
\begin{equation*}
\text{div}A(\phi _{n},x,Dv_{p,\phi _{n}})+B(\phi _{n},x,v_{p,\phi _{n}})=0
\end{equation*}
in the weak sense, where $A,B$ are defined by
\begin{eqnarray}
A\left( \phi ,x,\zeta \right) &=&F^{p-1}\left( \zeta J\right) F_{\xi}\left( \zeta
J\right) J^{T}       |\det \text{ }D\phi (x)|  \label{def_A_B} \\
B(\phi ,x,z) &=&\lambda _{p,\phi}|\det D\phi (x)||z|^{p-2}z  \notag
\end{eqnarray}
for all $x\in \overline{\Omega },\zeta \in \mathbb{R}^{N}\setminus
\left\{ 0\right\} ,\phi \in \Phi _{1,\alpha }\left( \Omega \right)$ and $z\in {\mathbb{R}}$.

Let $M_0>0$ be fixed.  We claim that there exists an open
neighborhood $\ \mathcal{V}$ $\ $of $\ \tilde{\phi \text{ }}$ and $c_{1},c_{2}\in (0,+\infty )$ such that
\begin{equation}
\sum\limits_{i,j=1}^{n}    \frac{\partial A_{i}\left( \phi ,x,\zeta
\right) }{\partial \zeta _{j}   }          \eta _{i}\eta
_{j}\geq c_{1}|\zeta |^{p-2}|\eta |^{2},  \label{Lem_H1}
\end{equation}

\begin{equation}
\left|   \frac{\partial A_{i}\left( \phi ,x,\zeta
\right) }{\partial \zeta _{j}}     \right|\leq c_{2}|\zeta |^{p-2},  \label{Lem_H2}
\end{equation}

\begin{equation}
|A\left( \phi ,x,\zeta \right) -A\left( \phi ,y,\zeta \right) |\leq c_{2}\left(
1+|\zeta|\right) ^{p-1}\left\vert x-y\right\vert ^{\alpha },  \label{Lem_H3}
\end{equation}

\begin{equation}
|B\left( \phi ,x,z\right) |\leq c_{2}\left( 1+|\zeta |\right) ^{p},
\label{Lem_H4}
\end{equation}
for all $\phi \in \mathcal{V},x,y\in \overline{\Omega },\zeta\in \mathbb{R}^{n}
\setminus \left\{ 0\right\} $ and $z\in \left( -M_{0},M_{0}\right) $.

%{\color{red}
%Indeed, we observe that
%\begin{equation*}
% \frac{\partial A_{i}\left( \phi ,x,\xi
%\right) }{\partial \xi _{j}}    = \left((p-1)F^{p-2}\left( \xi J\right) F_{\xi _{i}}\left( \xi J\right)+ F^{p-1}\left( \xi J\right)
%F_{\xi _{i}\xi _{j}}\left( \xi J\right) \right)  \left(
%\Gamma _{\phi }\right) _{i,j}|\det \text{ }D\phi (x)|
%\end{equation*}
%where $\left( \Gamma _{\phi }\right) _{i,j}$ is the $(i,j)-$ component of
%the matrix $\Gamma _{\phi }=(D\phi )^{-1}\cdot \left[ (D\phi )^{-1}\right]
%^{T}=J\cdot J^{T}.$}

Indeed, condition (\ref{Lem_H1}) follows by condition (\ref{ellip}) and a continuity
argument. Condition (\ref{Lem_H2}) follows observing that $F_{\xi }$ is
zero homogenous and using a continuity argument. Condition (\ref{Lem_H3}) and (\ref{Lem_H4})
follow using (\ref{eq:omo}) and reasoning as Lemma 3.1 of
\cite{Lamb}.

Clearly, $\phi_n\in {\mathcal{V}}$ for all $n\in \mathbb{N}$  sufficiently large, say for $n\geq \bar n$,  hence conditions
\eqref{Lem_H1}-\eqref{Lem_H4} are satisfied with $\phi$ replaced by $\phi_n$.
Thus,  by (\ref{L^infty_v}) we
are in position to apply the boundary regularity result in \cite{Lib}, which assures that there exists a constant $K>0$, such that

\begin{equation*}
v_{p,\phi _{n}}\in C^{1,\beta }(\overline{\Omega })\quad \text{and}\quad \left\Vert v_{p,\phi _{n}}\right\Vert _{1,\beta }\leq K,
\end{equation*}
for all $n\geq \overline{n}$. The thesis follows recalling that $C^{1,\beta}(\overline{\Omega})$ is compactly imbedded in $C^{1}(\overline{\Omega})$ and using standard regularity results.
\end{proof}

\section{Differentiability result and Hadamard formula}
\begin{thm}
\label{Frechet_diff}\bigskip Let $p>1$ and $\Omega $ be a bounded domain in $\mathbb{R}^{N}$ of class $C^{1,\alpha },$ with $\alpha
\in (0,1]$. Then the function which maps $\phi \in \Phi _{1,\alpha }\left(
\Omega \right) $ to $\lambda _{p,\phi}$ is of class $C^{1}.$
Moreover,
\begin{align}\label{d_egvalue}
  \text{\emph{d}}\lambda _{p,\phi}(\psi )= &  \int_{\phi (\Omega )}\left( F^{p}\left( Du_{p,\phi}\right) - \lambda
_{p,\phi }|u_{p,\phi}|^{p}\right)  \text{\emph{div}}\left( \psi \circ \phi ^{-1}\right)
dy \\
  & -p\int_{\phi (\Omega )}F^{p-1}\left( Du_{p,\phi}\right) F_{\xi }\left(
Du_{p,\phi}\right) \cdot  ( D\left( \psi \circ \phi ^{-1}\right)  )^T  \cdot \left(
Du_{p,\phi}\right) ^{T}dy, \nonumber
\end{align}
for all $\phi \in \Phi _{1,\alpha }\left( \Omega \right)$ and $\psi \in
C^{1,\alpha }\left( \overline{\Omega },\mathbb{R}^{N}\right) .$
\end{thm}
\begin{proof}  We consider the following real-valued functional
\begin{equation*}
R_{\phi }(v)=\frac{\int_{\Omega }F^{p}\left( Dv\cdot \left( D\phi \right)
^{-1}\right) |\det D\phi |dx}{\int_{\Omega }|v|^{p}|\det D\phi |dx},
\end{equation*}
defined for all $\left( v,\phi \right) \in C_{0}^{1}\left(
\overline{\Omega }  \right) \setminus \left\{ 0\right\}  \times \Phi _{1,\alpha
}\left( \Omega \right) $.
Since the map
\[
\phi \in \Phi _{1,\alpha }\left( \Omega \right) \rightarrow \left(
D\phi \right) ^{-1}\in C\left( \overline{\Omega },\mathbb{R}^{N\times N}\right)
\]
is real-analytic, by exploiting   the continuity of the functions $\phi \mapsto \partial _{\phi
}R\left( v \right) $ and $\phi \mapsto v_{p,\phi }$  (see Theorem \ref{Th_contuinity}),
we can use the same argument in \cite{Lamb} to prove  that  the map $\phi \mapsto \lambda _{p,\phi }$ is continuously
Frech\'{e}t differentiable and that its Frech\'{e}t differential at any point $\tilde \phi \in \Phi _{1,\alpha
}\left( \Omega \right)  $ is given by the formula
\begin{equation*}
\text{d}\lambda _{p,\tilde{\phi} }(\psi )=\partial _{\phi }R_{\tilde{\phi \text{
}}}(v_{p,\tilde{\phi}})(\psi )
\end{equation*}
which is valid for all  $\psi \in C^{1,\alpha
}\left( \overline{\Omega },\mathbb{R}^{N}\right).$

Now we fix $\psi \in C^{1,\alpha }\left( \overline{\Omega },\mathbb{R}^{N}\right) $ and we compute $\partial _{\phi }R_{\tilde{\phi \text{
}}}(v_{p,\tilde{\phi \text{ }}})(\psi ).$ To this end let $\delta >0$ be
such that $\phi _{t}:= \tilde{\phi}+t\psi \in \Phi
_{1,\alpha }\left( \Omega \right) ,$ for all $t\in (-\delta ,\delta ).$
Putting $R_{\phi _{t}}(v_{p,\tilde \phi  }   )=\frac{N}{D}$  where
\begin{equation*}
N=\int_{\Omega }F^{p}\left( Dv_{p,\tilde \phi   }\cdot \left( D\phi _{t}\right)
^{-1}\right) |\det D\phi _{t}|dx,
\end{equation*}
and
\begin{equation*}
D=\int_{\Omega }|v_{p,\tilde \phi    }|^{p}|\det D\phi _{t}|dx,
\end{equation*}
by standard calculus  we obtain
\begin{equation*}
\left.\frac{\text{d}D}{\text{dt}}\right\vert _{t=0}=\int_{\tilde{\phi} \left( \Omega \right)
}|u_{p,\tilde{\phi}}|^{p} \,\text{div}\chi dy,
\end{equation*}
and
\begin{align*}
\left.\frac{\text{d}N}{\text{dt}} \right\vert _{t=0}=& -p\int_{\tilde{\phi}\left( \Omega
\right) }F^{p-1}(Du_{p,\tilde{\phi}})F_{\xi }(Du_{p,\tilde{\phi}})  \cdot (D\chi )^{T} \cdot
\left( Du_{p,\tilde{\phi}}\right) ^{T}dy \\
&+\int_{\tilde{\phi}\left( \Omega \right) }F^{p}(Du_{p,\tilde{\phi}})\, \text{div}
\chi dy,
\end{align*}
where $\chi =\psi \circ \left(\tilde{\phi}\right)^{-1}$.
Putting all together and recalling \eqref{eigen_cond} and \eqref{eig}, formula \eqref{d_egvalue} follows.
\end{proof}
\bigskip

For any vector $n\in \mathbb{R}^N\setminus \{0\}$, we set
$$
n_F=F_{\xi }(n ),\quad \frac{\partial u}{\partial n_F}=D u \cdot n_F\, .
$$

A consequence of the previous result is the following generalization of the Hadamard formula to the anisotropic setting.

\begin{thm}
\label{Th_Hadamard} Let the same assumptions  of Theorem \ref{Frechet_diff} hold.
 Let $\phi$ $\in \Phi _{1,\alpha }\left( \Omega
\right) $ be such that $\phi( \Omega) $ is of class $C^{2,\delta }$, for some $\delta \in \left( 0,1\right] .$ Then
\begin{eqnarray}\label{Hadamard}
\text{\emph{d}}\lambda _{p,\phi}(\psi )& =& (1-p)\int_{\partial \phi( \Omega)}   \left| F(Du_{p,\phi})  \right|^p  \left( \psi \circ \phi^{-1}\right) \cdot \nu \, d\sigma \nonumber   \\
&=& (1-p)\int_{\partial \phi( \Omega)}   \left|\frac{\partial u_{p,\phi }}{\partial \nu_F}\right|^p  \left( \psi \circ \phi^{-1}\right) \cdot \nu \, d\sigma ,
\end{eqnarray}
for all $\psi \in C^{1,\alpha }\left( \overline{\Omega },\mathbb{R}^{N}\right)$, where $\nu$ is the unit outer normal to $\partial \phi (\Omega)$.
\end{thm}
\begin{proof} To shorten our notation, we set $u=u_{p,\phi }$. Recall that  by classical  boundary
regularity results, there exists $\beta \in (0,1)$ such that $u\in C^{1,\beta }\left( \phi\left( \overline{\Omega }\right) \right)$, hence the right hand side of (\ref{Hadamard}) is well defined, and $
u=0$ on $\partial \phi\left( \Omega \right)$. The possible lack of $C^2$ regularity of $u$ leads us  to consider
 a family of  approximating problems,  as often done in the literature (we refer to \cite{GS, lambre} for the case of the $p$-Laplacian, to
\cite{xiathesis} for the Finsler Laplacian and to \cite{ciraolo} for the  Finsler $p$-Laplacian).
As done in  \cite[Lemma~3.2]{ciraolo} and  \cite[Theorem~2.6]{xiathesis},
consider a family of sufficiently regular real-valued convex functions $V^{(\varepsilon )}$, with $\varepsilon>0$, defined in ${\mathbb{R}}^N$ such that $V^{(\varepsilon )} $ converges to  $F^p/p$ together with the first order derivatives,
 as $\varepsilon \to 0$. To be more specific, following \cite[Lemma~3.2]{ciraolo}, we take $
V^{(\varepsilon )} = (F^2+\varepsilon  )^{\frac{p}{2}}/p$, with $\varepsilon>0$.  Then we consider the  boundary value problem
\begin{equation}
\left\{
\begin{array}{ll}
-\text{div}\left(    V^{(\varepsilon)}_{\xi }\left( Du_{\varepsilon }\right) \right) =\lambda
_{p,\phi }|u|^{p-2}u, & \text{in }\phi\left( \Omega \right),  \\
u_{\varepsilon }=0, & \text{on }\partial \phi\left( \Omega \right) ,
\end{array}
\right.   \label{prob_approx}
\end{equation}
where  we recall that $V^{(\varepsilon)}_{\xi }$ denotes the gradient of $V^{(\varepsilon)}$. By the same proof  in \cite{ciraolo},
classical regularity results (see e.g. \cite{Ladyz}) assure that there
exists a solution $u_{\varepsilon }$ to problem (\ref{prob_approx}) such
that $u_{\varepsilon }\in  W^{2,2}(\phi (\Omega ))\cap C^{1,\gamma } (  \phi (\overline{\Omega}) )  $ for some $\gamma \in (0,1) $, and $\left\vert u_{\varepsilon }\right\vert \leq M_{0},$
with $M_{0}$ a constant independent of $\varepsilon .$ Moreover, the $C^{1,\gamma}$ estimates in \cite{Lib} imply that
%Applying the boundary
%regularity result due to Lieberman to problem (\ref{prob_approx}), there
%exist $\gamma \in (0,1)$ such that $u_{\varepsilon }\in C^{1,\gamma }\left(
%\phi\left( \overline{\Omega }\right) \right) $ and $\left\vert
%\left\vert u_{\varepsilon }\right\vert \right\vert _{1,\gamma }\leq M_{1}$,
%with $M_{1}$ a positive constant independent of $\varepsilon .$Moreover by
%the Ascoli-Arzel\`{a} theorem
$u_{\varepsilon }\rightarrow u$ in $
C^{1}\left(\phi(\overline\Omega) \right) $ up to
subsequences.

Now we approximate  the first term of the right-hand side of (\ref{d_egvalue}).
In order to shorten our notation we set
 $\chi =\psi \circ \left( \tilde{\phi }\right) ^{-1}.$
Using the Divergence Theorem and recalling that $u=0$
on $\partial \phi (\Omega ),$ we get
\begin{align*}
\int_{\phi (\Omega )}&\left(  p V^{(\varepsilon)}     \left( Du_{\varepsilon }\right)
         -\lambda _{p,\phi }|u|^{p}\right)
\text{div}\chi dy
\\
& =\int_{\partial \phi (\Omega )} pV^{(\varepsilon)}     \left( Du_{\varepsilon }\right)  \chi _{j}\nu
_{j}d\sigma -\int_{\phi (\Omega )}\frac{\partial }{\partial y_{j}}\left(p
V^{(\varepsilon)}     \left( Du_{\varepsilon }\right)   -\lambda _{p,\phi }|u|^{p}\right) \chi _{j}dy.
\end{align*}
Note that here and in the sequel the summation symbol is omitted.

Now we take into account the second term of the right hand side of (\ref{d_egvalue}). Reasoning as before, by the divergence
Theorem, we get
\begin{eqnarray}\label{Eq_lim} \lefteqn{
-p\int_{\phi (\Omega )} V^{(\varepsilon)}_{\xi }\left( Du_{\varepsilon }\right)      \cdot (D\chi )^T\cdot
\left( Du_{\varepsilon }\right) ^{T}dy
 =-p\int_{\partial \phi (\Omega
)}    V^{(\varepsilon)}_{\xi_i }\left( Du_{\varepsilon }\right)       \frac{\partial u_{\varepsilon }}{\partial y_{j}}\chi _{j}\nu _{i}d\sigma
} \\
&\nonumber \qquad\qquad\qquad\qquad\qquad\qquad\qquad\qquad\qquad\qquad\qquad\qquad+p\int_{\phi (\Omega )}\frac{\partial }{\partial y_{i}}\left(    V^{(\varepsilon)}_{\xi_i }\left( Du_{\varepsilon }\right)       \frac{\partial u_{\varepsilon }}{\partial y_{j}}
\right) \chi _{j}dy.
\end{eqnarray}

We first consider the second integral of right hand side of (\ref{Eq_lim}). Recalling that $u_{\varepsilon }$ are solutions to problem (\ref{prob_approx}), we obtain
\begin{align*}
\int_{\phi (\Omega )}&\frac{\partial }{\partial y_{i}}\left(   V^{(\varepsilon ) }_{\xi_i} (Du_{\varepsilon  })    \frac{\partial u_{\varepsilon }}{\partial y_{j}}    \right)
     \chi _{j}dy
 \\
 & =\int_{\phi (\Omega )}\!\!\!\text{div}\left(       V^{(\varepsilon ) }_{\xi} (Du_{\varepsilon  }) \right) \frac{\partial u_{\varepsilon }}{\partial
y_{j}}\chi _{j}dy \!+\!\!\int_{\phi (\Omega )} \!\!     V^{(\varepsilon )}_{\xi_i}      (Du_{\varepsilon  })                \frac{\partial ^{2}u_{\varepsilon }}{\partial y_{j}\partial y_{i}}
 \chi _{j}dy
\\
&=-\lambda _{p,\phi }\int_{\phi (\Omega )}\left\vert u\right\vert ^{p-2}u
\frac{\partial u_{\varepsilon }}{\partial y_{j}}\chi _{j}dy
+
\int_{\phi (\Omega )}\frac{\partial    V^{(\varepsilon )}     (Du_{\varepsilon  })      }{\partial y_{j}}        \chi _{j}dy.
\end{align*}

Substituting in (\ref{Eq_lim}), we get
\begin{align*}
-p\int_{\phi (\Omega )}&     V^{(\varepsilon ) }_{\xi} (Du_{\varepsilon  })      \cdot (D\chi )^T\cdot
\left( Du_{\varepsilon }\right) ^{T}dy\! =-p\!\!\int_{\partial \phi (\Omega
)}\!        V^{(\varepsilon ) }_{\xi} (Du_{\varepsilon  })   \frac{\partial u_{\varepsilon }}{
\partial y_{j}}\chi _{j}\nu _{i}d\sigma  \\
&-p\lambda _{p,\phi }\int_{\phi (\Omega )}\left\vert u\right\vert ^{p-2}u
\frac{\partial u_{\varepsilon }}{\partial y_{j}}\chi _{j}dy+   p \int_{\phi
(\Omega )}\frac{\partial    V^{(\varepsilon ) } (Du_{\varepsilon  })   }{\partial y_{j}}     \chi _{j}dy.
\end{align*}
It follows that
\begin{eqnarray}\label{astana}
\lefteqn{  \int_{\phi (\Omega )}\left(    p V^{(\varepsilon ) } (Du_{\varepsilon  })      -\lambda _{p,\phi }|u|^{p}\right)
\text{div}\chi dy}\nonumber   \\
& &-p\int_{\phi (\Omega )}   V^{(\varepsilon ) }_{\xi} (Du_{\varepsilon  })       \cdot (D\chi )^T\cdot
\left( Du_{\varepsilon }\right) ^{T}dy
\nonumber \\
& & =\int_{\partial \phi (\Omega )}      pV^{(\varepsilon ) }(Du_{\varepsilon  })\chi _{j}\nu
_{j}d\sigma + \lambda _{p,\phi }\int_{\phi (\Omega )}\frac{\partial }{\partial y_{j}}\left(
       |u|^{p}\right) \chi _{j}dy \nonumber  \\
& & -p\!\!\int_{\partial \phi (\Omega
)}\!   V^{(\varepsilon ) }_{\xi_i} (Du_{\varepsilon  })     \frac{\partial u_{\varepsilon }}{
\partial y_{j}}\chi _{j}\nu _{i}d\sigma  -p\lambda _{p,\phi }\int_{\phi (\Omega )}\left\vert u\right\vert ^{p-2}u
\frac{\partial u_{\varepsilon }}{\partial y_{j}}\chi _{j}dy.
\end{eqnarray}

Taking $\varepsilon \rightarrow 0$ in \eqref{astana},  we have
\begin{eqnarray}\label{penultima}
\lefteqn{  \int_{\phi (\Omega )}\left(  F^p \left( Du \right)
      -\lambda _{p,\phi }|u|^{p}\right)
\text{div}\chi dy}\nonumber   \\
& &-p\int_{\phi (\Omega )} F^{p-1}\left( Du \right)     F_{\xi }\left( Du \right) \cdot (D\chi )^T\cdot
\left( Du \right) ^{T}dy
\nonumber \\
& & =\int_{\partial \phi (\Omega )} F ^{p } \left(
Du  \right)  \chi _{j}\nu
_{j}d\sigma + \lambda _{p,\phi }\int_{\phi (\Omega )}\frac{\partial }{\partial y_{j}}\left(
       |u|^{p}\right) \chi _{j}dy \nonumber  \\
& & -p\!\!\int_{\partial \phi (\Omega
)}\!  F ^{p-1}\left( Du  \right)  \!F_{\xi_{i}}\left( Du  \right) \frac{\partial u  }{
\partial y_{j}}\chi _{j}\nu _{i}d\sigma  \nonumber   -p\lambda _{p,\phi }\int_{\phi (\Omega )}\left\vert u\right\vert ^{p-2}u
\frac{\partial u  }{\partial y_{j}}\chi _{j}dy\\
& & =\int_{\partial \phi (\Omega )} F^{p } \left(
Du  \right)  \chi _{j}\nu
_{j}d\sigma  -p\!\!\int_{\partial \phi (\Omega
)}\!   F ^{p-1}\left( Du  \right)  \!F_{\xi_{i}}\left( Du  \right) \frac{\partial u  }{
\partial y_{j}}\chi _{j}\nu _{i}d\sigma .
\end{eqnarray}

Now since $u=0$ on $\partial \phi (\Omega )$ and $u\in C^{1}\left(
\phi\left( \overline{\Omega }\right) \right)$, it follows that $\left. Du\right\vert _{\partial \phi (\Omega )}=\frac{\partial u}{\partial
\nu }\nu ,$ so by the homogeneity of $F$ it follows that

\begin{equation}
-p\int_{\partial \phi (\Omega )}F^{p-1}\left( Du\right) F_{\xi _{i}}\left(
Du\right) \frac{\partial u}{\partial y_{j}}\chi _{j}\nu _{i}d\sigma
=-p\int_{\partial \phi (\Omega )}F^{p}\left( Du\right) \chi \cdot \nu
d\sigma .  \label{B}
\end{equation}

\begin{comment}
\begin{align} \nonumber
\int_{\phi (\Omega )}&F^{p-1}\left( Du\right) F_{\xi _{i}}\left( Du\right)
\cdot D\chi \cdot \left( Du\right) ^{T}dy=\int_{\partial \phi (\Omega
)}F^{p}\left( Du\right) \chi \cdot \nu d\sigma\\
& +\int_{\phi (\Omega )}\frac{\partial }{\partial y_{j}}(F^{p}\left( Du\right) -\lambda _{p,\phi
}\left\vert u\right\vert ^{p})\chi _{j}dy.  \label{B}
\end{align}
\end{comment}

Finally, by Theorem \ref{Frechet_diff}, using equalities (\ref{penultima}) and (\ref{B}),
we obtain

\begin{equation}\label{Had_1}
\text{d}\lambda _{p,\phi}(\psi )=(1-p)\int_{\partial \phi\left( \Omega \right) }F^{p}\left( Du\right) \chi \cdot \nu \text{ }d\sigma
\end{equation}
which is the first equality in \eqref{Hadamard}.
Recalling that $D u = \frac{\partial u }{\partial \nu} \nu$,  it follows that
$$
F^p(  Du   )=F^p\left(  \frac{\partial u  }{\partial \nu} \nu  \right)=\left|  \frac{\partial u }{\partial \nu} F(\nu )   \right|^p=
\left|  \frac{\partial u }{\partial \nu} F_{\xi}(\nu ) \cdot \nu  \right|^p=\left|\frac{\partial u }{\partial \nu_F}\right|^p  ,
$$
which combined with \eqref{Had_1}  provides the validity of the second equality in \eqref{Hadamard}.
\end{proof}

\section{Corollaries}

In this section we briefly discuss two immediate corollaries of the previous section.

\subsection{Overdetermined problem}

As in the Euclidean case,  formula (\ref{Hadamard}) naturally leads to the formulation of an overdetermined
problem. Indeed, let us assume that the assumptions of Theorem \ref{Th_Hadamard} hold and let us consider the functionals $J_{1}$ and $J_2$ defined on  $ \Phi
_{1,\alpha }\left( \Omega \right)$ by $J_1(\phi ) = \lambda _{p,\phi }  $ and $J_2(\phi )= |\phi (\Omega)|$ for all $\phi \in \Phi
_{1,\alpha }\left( \Omega \right)$. If $\phi$ is a critical point for $J_{1}$ under the volume
constraint $|\phi (\Omega )|=c$ for a fixed constant $c>0$, then   ${\rm Ker }d J_2(\phi ) \subset {\rm Ker }d J_1(\phi )  $, hence   $d J_1(\phi )=\kappa d J_2(\phi )$ for some  $\kappa \in \mathbb{R}$ (the Lagrange multiplier).   This means that
\begin{equation*}
(1-p)\int_{\partial \phi (\Omega )}\left|\frac{\partial u_{p,\phi} }{\partial \nu_F}\right|^p  \psi \circ \left( \tilde{\phi }\right) ^{-1} \cdot \nu d\sigma =\kappa \int_{\partial \phi (\Omega )}\psi \circ \left( \tilde{\phi }\right) ^{-1}
\cdot \nu \text{ }d\sigma   ,
\end{equation*}
for all $\psi \in C^{1,\alpha
}\left( \overline{\Omega },\mathbb{R}^{N}\right)$.

\begin{comment}
Thus, by the arbitrary choice of $\psi $ we deduce that $F^{p}\left( Du_{p,\phi} \right)  $ is constant of $\partial \phi (\Omega)$.  We now set
$$
\nu_F=F_{\xi }(\nu ),\quad \frac{\partial u}{\partial \nu_F}=D u \cdot \nu_F\, .
$$

Recalling that $D u_{p,\phi}= \frac{\partial u_{p,\phi }}{\partial \nu} \nu$,  it follows that
$$
F^p(  Du_{p,\phi}   )=F^p\left(  \frac{\partial u_{p,\phi }}{\partial \nu} \nu  \right)=\left|  \frac{\partial u_{p,\phi }}{\partial \nu} F(\nu )   \right|^p=
\left|  \frac{\partial u_{p,\phi }}{\partial \nu} F_{\xi}(\nu ) \cdot \nu  \right|^p=\left|\frac{\partial u_{p,\phi }}{\partial \nu_F}\right|^p  .
$$
\end{comment}

Thus, $\left|\frac{\partial u_{p,\phi} }{\partial \nu_F}\right|^p =\kappa /(1-p)$  on $\partial \phi (\Omega)$, and this implies that $\frac{\partial u_{p,\phi }}{\partial \nu_F}$
is also constant because the function $u_{p,\phi}$ doesn't change sign.

In conclusion,  $\phi$ is a critical point for $J_{1}$ under the  volume constraint above if and only if
 $\frac{\partial u_{p,\phi }}{\partial \nu_F}$ is constant on $\partial \phi (\Omega)$. In other words,
 $ u_{p,\phi }$ solves the following overdetermined problem
\begin{equation} \label{overdet}
\left\{
\begin{array}{ll}
-\text{div}\left( F^{p-1}\left( Du_{p,\phi}\right) F_{\xi }\left(
Du_{p,\phi }\right) \right) =\lambda _{p,\phi }\left\vert u_{p,
\phi}\right\vert ^{p-2}u_{p,\phi},& \text{ in } \phi (\Omega),\vspace{1mm} \\
u_{p,\phi}=0,& \text{ on } \partial\phi\left( \Omega \right),\vspace{1mm} \\
\frac{\partial }{\partial \nu_F }u_{p,{\phi}}= \text{constant},&  \text{ on } \partial\phi\left( \Omega \right).
\end{array}
\right.
\end{equation}
System \eqref{overdet} is  satisfied  when   $\phi\left( \Omega \right)$ is homotetic to a
Wulff shape. This can be deduced either by the symmetry result  in  \cite[Theorem 2.4]{dgmana} or by  the Faber-Krahn inequality
which assures that if $\mathcal{W}$ is the Wulff shape centered in the origin such that $| \mathcal{W}|=|
\phi\left( \Omega \right) | ,$ then $\lambda _{p}(\mathcal{W})\leq
\lambda _{p}({\phi}\left( \Omega \right))$, see e.g. \cite{bfk}.
On the other hand, it is proved in \cite{wxpac2011} that if  system  \eqref{overdet} is  satisfied  then $\Omega$ is homotetic to a Wulff shape. We note that the same conclusion holds for the overdetermined system
\begin{equation} \label{overdetbis}
\left\{
\begin{array}{ll}
-\text{div}\left( F^{p-1}\left( Du  \right) F_{\xi }\left(
Du  \right) \right) =1,& \text{ in } \phi (\Omega),\vspace{1mm} \\
u=0,& \text{ on } \partial\phi\left( \Omega \right),\vspace{1mm} \\
\frac{\partial }{\partial \nu_F }u= \text{constant},&  \text{ on } \partial\phi\left( \Omega \right),
\end{array}
\right.
\end{equation}
as it is has been proved in \cite{CS}.

\subsection{Rellich-Pohozaev identity}

Given  a bounded domain $\Omega $ be  in $\mathbb{R}^{N}$ of class $C^{2,\alpha },$ with $\alpha
\in (0,1]$, we consider a family of dilations  $(1+t)\Omega$  of $\Omega$ which can viewed as a family of  diffeomorphisms  $\phi_t=I+tI$, $t\in {\mathbb{R}}$.

By  differentiating with respect to $t$ and applying formula (\ref{Hadamard}) with $\phi = \psi=I$,  we obtain
\begin{equation}\label{rellich1}
\left. \frac{d}{dt}\lambda _{p,\phi _{t}}\right\vert
_{t=0}=(1-p)\int_{\partial \Omega }\left|\frac{\partial u_{p,\phi }}{\partial \nu_F}\right|^p x \cdot \nu d\sigma .
\end{equation}

 By Lemma~\ref{monoto} we get that
  $\lambda _{p,\phi _{t}}=\lambda _{p}\left( \left( 1+t\right)
\Omega \right) =\left( 1+t\right) ^{-p}\lambda _{p}\left( \Omega \right) $. If we derive this last expression with respect to $t$, we obtain
\begin{equation}\label{rellich2}
\left. \frac{d}{dt}\lambda _{p,\phi _{t} }\right\vert
_{t=0}=\left. \frac{d}{dt}\left( (1+t)^{-p}\lambda _{p}(\Omega)\right)\right\vert _{t=0}=-p\lambda _{p}(\Omega).
\end{equation}

Combining \eqref{rellich1} and \eqref{rellich2}, we  obtain the following Rellich-Pohozaev identity
\begin{equation}\label{rellichformula}
\lambda _{p}(\Omega )=\frac{p-1}{p}\int_{\partial \Omega } \left|\frac{\partial u_{p,\phi }}{\partial \nu_F}\right|^p    x \cdot \nu d\sigma .
\end{equation}

Formula \eqref{rellichformula} holds not only for the first eigenvalue but for any eigenvalue. In fact, the following theorem holds.

\begin{thm}
 Let $p>1$ and $\Omega $ be a bounded domain in $\mathbb{R}^{N}$ of class $C^{2,\alpha },$ with $\alpha \in (0,1]$.
 Let $\lambda$ be an eigenvalue of equation \eqref{intro_eigpb} and $u$ be a corresponding eigenfunction normalized by
 $\| u\|_{L^p(\Omega)}=1$. Then
 \begin{equation*}\label{rellichformulabis}
\lambda =\frac{p-1}{p}\int_{\partial \Omega } \left|\frac{\partial u}{\partial \nu_F}\right|^p    x \cdot \nu d\sigma .
\end{equation*}
\end{thm}

Since the Hadamard formula for the Finsler Laplacian is currently proved only for the first eigenvalue, the proof of the previous theorem in the case of an arbitrary eigenvalue cannot be performed as above. However,  one can adapt the same argument used in \cite{lambre}  for the $p$-Laplacian and based on   the original approach of F. Rellich~\cite{rellich}.
\begin{proof}
As done in the proof of Theorem \ref{Th_Hadamard} one may follow the approximation argument of
\cite[Lemma~3.2]{ciraolo} or  \cite[Theorem~2.6]{xiathesis} and consider the approximating family of boundary value problems \eqref{prob_approx} defined in $\Omega$.  Reasoning as \cite{ciraolo}, classical regularity results (see e.g. \cite{Ladyz}) assure that there
exists a solution $u_{\varepsilon }$ to problem (\ref{prob_approx}) such
that $u_{\varepsilon }\in  W^{2,2}(\Omega )\cap C^{1,\gamma } (  \overline{\Omega})  $ for some $\gamma \in (0,1) $, and $\left\vert u_{\varepsilon }\right\vert \leq M_{0},$
with $M_{0}$ a constant independent of $\varepsilon .$ Moreover, the $C^{1,\gamma}$ estimates in \cite{Lib} imply that
%Applying the boundary
%regularity result due to Lieberman to problem (\ref{prob_approx}), there
%exist $\gamma \in (0,1)$ such that $u_{\varepsilon }\in C^{1,\gamma }\left(
%\phi\left( \overline{\Omega }\right) \right) $ and $\left\vert
%\left\vert u_{\varepsilon }\right\vert \right\vert _{1,\gamma }\leq M_{1}$,
%with $M_{1}$ a positive constant independent of $\varepsilon .$Moreover by
%the Ascoli-Arzel\`{a} theorem
$u_{\varepsilon }\rightarrow u$ in $
C^{1}\left(\overline\Omega \right) $ up to
subsequences.
Reasoning as in \cite{lambre}, one can multiply both sides of the first equality in \eqref{prob_approx} by
$\sum_{k=1}^{N}x_k \frac{\partial u_{\varepsilon}}{\partial x_k}$ and integrating, it follows that
\begin{equation}
-\int_{\Omega} x_k \frac{\partial}{\partial x_h}\left(V^{(\varepsilon ) }_{\xi_{h }   } (Du_{\varepsilon  })\frac{\partial u_{\varepsilon}}{\partial x_k}\right)+
\int_{\Omega} x_k \frac{\partial}{\partial x_k}(V^{(\varepsilon ) } (Du_{\varepsilon  }))=
\lambda \int_{\Omega} |u|^{p-2}u  x_k   \frac{\partial u_{\epsilon }}{\partial x_k} dx.
\end{equation}
Here and in the sequel the summation symbol is omitted. By applying the Divergence Theorem to both sides of the previous equality and letting $\varepsilon\rightarrow 0$, we get
\begin{equation}
\begin{split}
  -\int_{\partial \Omega} & x_k F^{p-1}(Du)F_{\xi_ {h }}(Du)  \frac{\partial u}{\partial x_k}   \nu_h d\sigma+\int_\Omega F^{p-1}(Du)F_{\xi_{ k}}(Du) \frac{\partial u}{\partial x_k}
   \\
  &+\frac{1}{p}\int_{\partial \Omega} x_k F^p(Du) \nu_k d\sigma -\frac{N}{p} \int_\Omega F^p(Du) dx= -\frac{N\lambda}{p},
  \end{split}
\end{equation}
where the last equality follows since $\| u\|_{L^p(\Omega)}=1$. Finally, the statement follows using the properties of $F$ and recalling that $F^p(Du)=\left|\frac{\partial u }{\partial \nu_F}\right|^p$. \\
\end{proof}

{\bf Acknowledgments:} The authors are very thankful to Prof. Enrique Zuazua for bringing to their attention the method which allows to deduce Rellich-type identities from  Hadamard-type formulas.  The authors are also very  thankful to an anonymous referee for the careful reading of the paper and for pointing out the approximating procedure in \cite[Theorem~2.6]{xiathesis}  which led them to simplify the presentation and adjust  the procedure in the proof of formula \eqref{Hadamard}.
The authors are members of the Gruppo Nazionale per l'Analisi Matematica, la Probabilit\`a e le loro Applicazioni (GNAMPA) of the Istituto Nazionale di Alta Matematica (INdAM). Research partially supported by project Vain-Hopes within the program VALERE: VAnviteLli pEr la RicErca.

\end{document}